\documentclass[a4paper,english,reqno]{amsart}

\usepackage{amssymb,amsthm, amsmath,amsfonts}
\usepackage{latexsym}
\usepackage{bbm}

\usepackage{pgf,tikz}
\usetikzlibrary{arrows,shapes,calc,positioning,matrix}
\usepackage{tikz-cd}

\usepackage[cmtip,all]{xy}
\xyoption{arc}

\usepackage{graphicx}
\usepackage{subfigure}

\usepackage{url}
\usepackage{microtype}

\usepackage{xcolor}
\usepackage{hyperref}

\newcommand{\Z}{\mathbb{Z}}
\newcommand{\R}{\mathbb{R}}

\newcommand{\PSL}{{\rm PSL}(2,\R)}
\newcommand{\SL}{{\rm SL}(2,\R)}
\newcommand{\SLn}{{\rm SL}(n+1,\R)}
\newcommand{\PSLn}{{\rm PSL}(n+1,\R)}
\renewcommand{\P}{\mathbb{R}{\rm P}}

\renewcommand{\H}{\mathbb{H}}

\newcommand{\F}{\mathcal{F}}
\newcommand{\M}{\mathcal{M}}

{

\newcommand{\matriz}[4]{\left(
\begin{array}{ll}
 #1 & #2 \\
 #3 & #4
\end{array}
\right)}

\newcommand{\norm}[1]{\parallel \!#1 \! \parallel}

\newcommand{\bs}{\backslash}
\newcommand{\scirc}{{\scriptscriptstyle \circ}}

\newcommand{\widebar}[1]{\mkern 1.7mu\overline{\mkern-1.7mu#1\mkern-1.7mu}\mkern 1.7mu}

\newtheorem{theorem}{Theorem}
\newtheorem{lemma}{Lemma}
\newtheorem{proposition}{Proposition}
\newtheorem{corollary}{Corollary}
\newtheorem*{theorem*}{Theorem}
\newtheorem{dualtheorem}{Theorem}

\theoremstyle{definition}
\newtheorem{definition}{Definition}
\newtheorem{example}{Example}
\newtheorem{remark}{Remark}

\newtheoremstyle{named}{}{}{\itshape}{}{\bfseries}{.}{.5em}{\thmnote{#3} #1}
\theoremstyle{named}

\begin{document}

\title[Horocycle flow on flat projective bundles]{Horocycle flow on flat projective bundles: \\ topological remarks and applications}

\author[F. Alcalde]{Fernando Alcalde Cuesta}
\address{\small ~Universidad de Santiago de Compostela, E-15782 Santiago de Compostela (Spain).}
\email{fernando.alcaldecuesta@gmail.com}

\author[F. Dal'Bo]{Fran\c{c}oise Dal'Bo}
\address{\small ~Institut de Recherche Math\'ematique de Rennes, Universit\'e de Rennes 1, F-35042 Rennes (France)}
\email{francoise.dalbo@univ-rennes1.fr}



\date{}

\dedicatory{}

\begin{abstract}
In this paper we study topological aspects of the dynamics of the foliated horocycle flow on flat projective bundles over hyperbolic surfaces and we derive ergodic consequences. 
If $\rho : \Gamma \to \PSLn$ is a representation of a non-elementary Fuchsian group $\Gamma$, the unit tangent bundle $Y$ associated to the flat projective bundle defined by $\rho$ admits a natural action of the affine group $B$ obtained by combining the foliated geodesic and horocycle flows. 
If the image $\rho(\Gamma)$ satisfies \emph{Conze-Guivarc'h conditions}, namely strong irreducibility and proximality, the dynamics of the $B$-action is captured by the proximal dynamics of $\rho(\Gamma)$ on $\P^n$ (Theorem A). 
In fact, the dynamics of the foliated horocycle flow on the unique $B$-minimal subset of $Y$ can be described in terms of dynamics of the horocycle flow on the non-wandering set in the unit tangent bundle $X$ of the surface $S= \Gamma \bs \H$ (Theorem B). 
Assuming the existence of a continuous limit map, we prove that the $B$-minimal set is an attractor for the foliated horocycle flow restricted to the proximal part of the non-wandering set in $Y$ (Theorem C). As a corollary, we deduce that the restricted flow admits a unique conservative ergodic $U$-invariant Radon measure (defined up to a multiplicative constant) if and only if $\Gamma$ is convex-cocompact. For example, the foliated horocycle flow on the 
sphere bundle defined by the Cannon-Thurston map is uniquely ergodic.

\end{abstract}

\maketitle

\section{Introduction} \label{SIntro}

In the 1930s G.A. Hedlund \cite{Hedlund} proved the minimality of the right action of the unipotent subgroup
$$
U = \{ \, \matriz{1}{s}{0}{1} \, | \; s \in \R \, \} 
$$
of $\PSL = \{\pm Id\} \bs \SL $ on the quotient $X = \Gamma \bs \PSL$ by a cocompact Fuchsian group $\Gamma$. Later H. Furstenberg \cite{Furstenberg2} obtained a stronger result, namely the $U$-action is 
uniquely ergodic. 
Identifying $\PSL$ and the unit tangent bundle of the hyperbolic plane $\H$ with the Poincar\'e metric, when $\Gamma$ is torsion-free, the quotient $X$ becomes the unit tangent bundle of the hyperbolic surface $S = \Gamma \bs \H$. In this geometric setting, the $U$-action on $X$ identifies with the \emph{horocycle flow} and we write
$$
h_s(\Gamma u)  = \Gamma u \matriz{1}{s}{0}{1}
$$ 
for all $u \in \PSL$ and all $s \in \R$.
Hedlund's and Furstenberg's results have been extended to the case where $\Gamma$ is finitely generated, but replacing $X$ by the non-wandering set $\Omega_X$ of the $U$-action. Notice that $\Omega_X$ is also the unique non-empty minimal invariant closed set for the action of the affine group  
$$
B = \{ \, \matriz{a}{b}{0}{a^{-1}}  \,  | \ b \in \R , a \in \R^+_\ast \, \}
$$
on $X$. In that case, the dynamic properties of the U-action from a double topological and measurable perspective can be gathered in the following statement:

\begin{theorem*}
Let $\Gamma$ be a finitely generated Fuchsian group. 
\medskip

\noindent
(1) For any $x \in \Omega_X$, either $xU$ is periodic or $\overline{xU} = \Omega_X$ \cite{Dal'Bobook,Eberlein,Hedlund}. 
\medskip

\noindent
(2) For any ergodic $U$-invariant Radon measure $\mu$ supported by $\Omega_X$, either $\mu$ is supported by a periodic orbit or $\mu$ is the Burger-Roblin measure up to a multiplicative constant \cite{Burger,Ratner1990,Ratner1992,Ratner1994,Roblin}.
\end{theorem*}

As explained in \cite{Dal'Bobook}, it turns out that property (1) is true if and only if $\Gamma$ is finitely generated. However, the topological dynamics of the $U$-action on $\Omega_X$ is not well understood otherwise. On the other hand, it follows from Ratner's work that the measure $\mu$ in property (2) is finite if and only if $\mu$ supported by a periodic orbit or $\mu$ is the Haar measure (up to a constant) and in this case the surface $S$ has finite volume. 
\medskip 

In this paper we study the foliated horocycle flow on flat projective bundles over hyperbolic surfaces. Given a non-elementary Fuchsian group  $\Gamma$, we consider a representation 
$$
\rho : \Gamma \to \PSLn
$$
with $n \geq 1$. The subgroup $\Gamma_\rho = \{ \,\big(\gamma,\rho(\gamma)\big) \, | \, \gamma \in \Gamma \, \}$ of \mbox{$\PSL \times \PSLn$} acts properly discontinuously on $\tilde Y = \PSL \times \P^n$. As this action preserves the product structure of $\tilde Y$, the projective bundle $Y = \Gamma_\rho \bs \tilde Y$ over $X = \Gamma \bs \PSL$ admit a foliation transverse to the fibration $\pi : Y \to X$ 
(which sends $\Gamma_\rho(u,\chi) \in Y$ onto $\Gamma u \in X$). The leaves are 3-manifolds endowed with a natural $\PSL$-geometric structure. Clearly the $U$-action on $\tilde Y$ defined by right translation on the first factor induces an $U$-action on $Y$ preserving each leaf. This action defines the \emph{foliated horocycle flow} on $Y$ 
 \cite{Martinez-Matsumoto-Verjovsky}.
 In the same way, the affine group $B$ acts on $Y$ preserving each leaf. 
\medskip 

Our goal is to prove topological properties of the actions of $B$ and $U$ on $Y$ when $\rho$ satisfies two conditions, which we call \emph{Conze-Guivarc'h conditions}:
\medskip

\noindent
(CG1) $\rho(\Gamma)$ is strongly irreducible, 
\medskip 

\noindent
(CG2) $\rho(\Gamma)$ contains a proximal element.
\medskip 

\noindent
Both conditions guarantee the existence of a unique non-empty minimal $\rho(\Gamma)$-invariant closed set $L(\rho(\Gamma))$ in $\P^n$ \cite{Conze-Guivarc'h}. It is   the closure of the dominant directions of the proximal elements of $\rho(\Gamma)$.
\medskip

The following results extend well known properties of the actions of $B$ and $U$ on $X$ to the projective bundle $Y$:
\smallskip 

\begin{theorem} \label{thm:thmA} 
Let $\Gamma$ be a non-elementary Fuchsian group and $\rho : \Gamma \to \PSLn$ be a representation satisfying conditions {\rm (CG1)} and {\rm (CG2)}. Then there is a unique non-empty minimal $B$-invariant closed set $\M_B \subset Y$ (or $B$-minimal set in short), that is, for every point $y \in Y$, the orbit closure $\overline{yB} \supset \M_B$. 
\end{theorem}
\smallskip 

\begin{theorem} \label{thm:thmB} 

Under the same assumptions of Theorem~\ref{thm:thmA}, for each point $y \in \M_B$, we have: 
 $$
\overline{yU} = \M_B \; \Leftrightarrow \; \overline{\pi(y)U} = \Omega_X.
$$
\end{theorem}

\begin{corollary} \label{cor:finitelygenerated}
Let $\Gamma$ be a non-elementary Fuchsian group and $\rho : \Gamma \to \PSLn$ be a representation satisfying conditions {\rm (CG1)} and {\rm (CG2)}. 
Then $\M_B$ is $U$-minimal (and therefore the unique $U$-minimal subset of $Y$) if and only if $\Gamma$ is convex-cocompact.
\end{corollary}

 In the last part of the paper, we will add a strong condition on $\rho$, called \emph{Nielsen's condition}, implying the existence of a continuous section for $\pi$:
\medskip 

\noindent
 {\rm (N)}  $\rho$ induces a continuous map $\varphi : L(\Gamma) \to L(\rho(\Gamma))$,
 called \emph{limit map}, such that $\varphi \scirc \gamma = \rho(\gamma) \scirc \varphi$ for all $\gamma \in \Gamma$. 
\medskip 

\noindent
As $L(\rho(\Gamma))$ is minimal, the map $\varphi$ is always surjetive. If we denote
$$
Y_{prox} = \Gamma_\rho \bs \PSL \times L(\rho(\Gamma)),
$$ 
the $B$-invariant closed set $\Omega_{prox} =  Y_{prox} \cap \pi^{-1}(\Omega_X)$ inherits from $Y$ a natural structure of $L(\rho(\Gamma))$-fibre bundle over $\Omega_X$. By construction, it always contains the $B$-minimal set $\M_B$. Condition {\rm (N)} gives arise to a continuous section $\Phi : X \to Y$ for the fibration $\pi  : Y \to X$. 

\begin{theorem} \label{thm:thmC}
Let $\Gamma$ be a non-elementary Fuchsian group and \mbox{$\rho : \Gamma \to \PSLn$} be a representation satisfying conditions {\rm (CG1)}, {\rm (CG2)} and {\rm (N)}. Then 
\medskip 

\noindent
(1) $\M_B  = \Phi (\Omega_X)$, 
\medskip 

\noindent
(2) $\M_B$ is a $U$-attractor relative to $\Omega_{prox}$, i.e. for any point $y \in \Omega_{prox}$ and for any sequence $s_k \to \pm\infty$, we have: 
$$
h_{s_k}(y) \to y' ~ \Rightarrow ~ y' \in \M_B.
$$ 
\end{theorem}

\begin{corollary} \label{cor:onetoone}
Under the conditions of Theorem~\ref{thm:thmC}, if $m$ is a conservative ergodic $U$-invariant Radon measure on $Y$ supported by $\Omega_{prox}$, then the support of $m$ is equal to $\M_B$ and there exist a conservative ergodic $U$-invariant Radon measure $\mu$ on $X$ supported by $\Omega_X$ such that $m = \Phi_\ast \mu$.
\end{corollary}

\begin{corollary} \label{cor:uniqerg}
Under the conditions of Theorem~\ref{thm:thmC}, assume $\Gamma$ is finitely generated. Then there is a unique conservative ergodic $U$-invariant Radon measure $m$ on $\Omega_{prox}$ (defined up to a multiplicative constant and supported by the unique $U$-minimal set $\M_B$ in $\Omega_{prox}$) if and only if $\Gamma$ is convex-cocompact. In particular, there is a unique $U$-invariant probability measure $m$ on $Y_{prox}$ if and only if $\Gamma$ is cocompact. 
\end{corollary}


The uniqueness of $U$-invariant probability measures on $Y$ projecting to Haar measure on $X$ has been proved by C. Bonatti, A. Eskin and A. Wilkinson \cite{Bonatti-Eskin-Wilkinson} when $\Gamma$ has finite covolume. Here we use Nielsen's condition (N) to deduce a similar property, both for finite and infinite measures, from the existence of a topological attractor. However, a strictly ergodic approach can be also applied.
Details will be discussed elsewhere.

\section{Preliminaries} \label{SPre}

A matrix $A \in {\rm SL}(n+1,\R)$ is said to be \emph{proximal} if $A$ admits a simple dominant real eigenvalue $\lambda_A$. Let $w_A \in \R^{n+1}$ be an eigenvector associated to $\lambda_A$ and $\chi_A \in \P^n$ its direction, also called \emph{dominant} for $A$. Further, we have the decomposition
$$
\R^{n+1} = \R w_A \oplus W_A
$$
where 
$$
W_A = \{ \, w \in \R^{n+1} \, | \, \lambda_A^{-k}A^k w \to 0 \text{ as } k \to +\infty \, \}.
$$

\begin{definition} \label{def:Conze-Guivarc'h}
Let $G$ be a subgroup of $\SLn$. We say $G$ is:
\medskip

\noindent
{\rm (CG1)} \emph{strongly irreducible} if there does no exist any proper non-trivial subspace of $\R^{n+1}$ invariant by the action of a subgroup of finite index of $G$; 
\medskip

\noindent
{\rm (CG2)} \emph{proximal} if $G$ contains a proximal element $A$. 
\medskip

\noindent
Both conditions will be called \emph{Conze-Guivarc'h conditions}.
\end{definition} 

\noindent
Conditions {\rm (CG1)} and {\rm (CG2)} are satisfied by $G$ if and only if they are satisfied by its Zariski closure in $\SLn$ \cite{Conze-Guivarc'h}. But these conditions do not imply that $G$ is Zariski dense in $\SLn$ (since ${\rm SO}(n,1)$ satisfies {\rm (CG1)} and {\rm (CG2)}).

\begin{proposition}[\cite{Conze-Guivarc'h}] \label{prop:CG} Let $G$ be a subgroup of $\SLn$ satisfying {\rm (CG1)} and {\rm (CG2)}. Then 
\[
L(G) = \overline{ \{ \, \chi_A \in \P^n \, | \, A \in G \text{ proximal} \, \} }. 
\] 
is the unique $G$-minimal set in $\P^n$. \qed
\end{proposition}
 
\begin{remark}
Assume $G \subset \PSLn$ is discrete and consider the $G$-action induced on $L^c(G) = \P^n - L(G)$. For $n=1$, as this action is properly discontinuous, the set $L(G)$ is 
a \emph{$G$-attractor} (i.e. for any point $\xi \in \P^1$ and for any non stationary sequence $g_k$ in $G$, the condition $g_k.\xi \to \xi'$ implies $\xi' \in L(G)$) which captures the proximal dynamics of $G$. However, for $n \geq 2$, these properties are not true in general as the following example proves.
\end{remark}

\begin{example} \label{ex:lorentz}
Consider $\R^3$ equipped with the Lorentz quadratic form 
$$
q(x) = x_1^2+x_2^2-x_3^2.
$$
For $i=-1,0,1$, denote $\mathcal{H}_i = \{ \, x \in \R^3  \,  | \, q(x) = \, i \}$ and let
$p : \R^3-\{0\} \to \P^2$ be the canonical projection. Let $SO^+(2,1)$ be the connected component of the identity in the group $SO(2,1)$ of orientation-preserving linear isometries of $q$ and take a discrete subgroup $G$ of $SO^+(2,1)$. If $G$ is non-elementary and contains no elliptic elements, then $G \bs \mathcal{H}_{-1}^+$ is isometric to a hyperbolic surface $S$, where $\mathcal{H}_{-1}^+ = \mathcal{H}_{-1} \cap \{ \,  x \in \R^3 \, | \, x_3 \geq 0 \, \}$ (see \cite{Dal'Bobook}). The limit set $L(G)$ is contained into $p(\mathcal{H}_0^+ -\{0\})$. 
For any vector $x \in \mathcal{H}_1$, the orthogonal plane (with respect to $q$) intersects 
$\mathcal{H}_0$ along two lines $D_1(x)$ and $D_2(x)$. Let $\mathcal{H}_1(G)$ be the set of vectors $x \in \mathcal{H}_1$ such that the directions of $D_1(x)$ and $D_2(x)$ belong to $L(G)$. This is a $G$-invariant closed subset of $\R^3-\{0\}$. 
According to \cite[Proposition VI.2.5]{Dal'Bobook}, the dynamics of the $G$-action on $\mathcal{H}_1(G)$ is dual to that of the geodesic flow on the non-wandering set of $T^1 S$. In particular,
the $G$-action on $\mathcal{H}_1(G)$ has dense orbits (see \cite[Property VI.2.12]{Dal'Bobook}), as well many non-empty proper minimal sets, and hence the $G$-action on the closure $\F(G)$ of $p(\mathcal{H}_1(G))$ in $\P^2$ also has dense orbits, as well many non-empty invariant closed sets $F \subset \F(G)$ such that $L(G) \subset F$. In conclusion, the $G$-action on $L^c(G)$ is not discontinuous and $L(G)$ is far from being a $G$-attractor. 
\end{example}

Let $\Gamma$ be a non-elementary Fuchsian group and $\rho : \Gamma \to \SLn$ a representa\-tion satisfying conditions {\rm (CG1)} and {\rm (CG2)}. Theorems~\ref{thm:thmA}~and \ref{thm:thmB} will be proved using a dual approach. Namely,  as the linear action of $\Gamma$ on $E =  \{ \pm Id \}\bs \R^2-\{0\}$ and the projective action of $\Gamma$ on $\P^1$ are conjugated to the $\Gamma$-actions on
$\PSL/U$ and $\PSL/B$ respectively, both actions are dual to the $U$-action and the $B$-action on $X = \Gamma \bs \PSL$. 
\medskip 

In our case, the linear and projective actions extend to actions of 
$$\Gamma_\rho = \{ \, \big(\gamma,\rho(\gamma)\big) \, | \, \gamma \in \Gamma \, \}$$ on $E \times \P^n$ and $\P^1 \times \P^n$. As before, they are dual to the $U$-action and the $B$-action on the flat projective bundle $Y= \Gamma_\rho \bs \tilde Y$ over $X$ where $\tilde Y = \PSL \times \P^n$. From a geometrical point of view, $Y$ is the unitary tangent bundle to the foliation by hyperbolic surfaces on $\Gamma_\rho \bs \H \times \P^n$ which is induced by the horizontal foliation on $\H \times \P^n$. 

\begin{dualtheorem} \label{thm:thmA*}
Under the assumptions of Theorem~\ref{thm:thmA}, there is a unique non-empty minimal $\Gamma_\rho$-invariant closed set $\M \subset \P ^1 \times \P^n$. Moreover $\M \subset L(\Gamma) \times L(\rho(\Gamma))$.
\end{dualtheorem}

The relation between the sets $\M_B$ and $\M$ considered in Theorems~\ref{thm:thmA}~and~\ref{thm:thmA*} is given by
$$
\M_B = \{ \, \Gamma_\rho(u, \chi) \in Y \, | \, (u(+\infty),\chi) \in \M \,  \},
$$
where $u(+\infty)$ is the endpoint of the geodesic ray associated to $u \in T^1\H$.
\medskip 

For each vector $v \in E$, denote $\bar v \in \P^1$ its direction. Clearly the set 
$$
E(\Gamma) = \{ \, v \in E \, | \, \bar v \in L(\Gamma) \, \}
$$
is dual to $\Omega_X$ and the set 
$$
E(\M) = \{ (v,\chi) \in E \times \P^n \, | \, (\bar v,\chi) \in \M \, \}.
$$
is dual to $\M_B$.

\begin{dualtheorem} \label{thm:thmB*}
Under the assumptions of Theorem~\ref{thm:thmA}, for each pair $(v,\chi) \in E(\M)$, we have: 
 $$
\overline{\Gamma_\rho (v,\chi)} = E(\M) \; \Leftrightarrow \; \overline{\Gamma v} =  E(\Gamma).
$$ 
\end{dualtheorem}

\section{Proof of Theorems~A*~and~B*} \label{SthmA&B}

Let $\Gamma$ be a non-elementary Fuchsian group and $\rho : \Gamma \to \PSLn$ be a representation satisfying {\rm (CG1)} and {\rm (CG2)}. Take $(\gamma,A) \in \Gamma_\rho$ with $A$ proximal and denote $\chi_A  \in \P^n$ the dominant direction of $A$. Since $A$ has infinite order, $\gamma$ is hyperbolic or parabolic. Consider $\gamma^+ = \lim_{k \to +\infty} \gamma^k(z)$ for any $z \in \H$. 

\begin{lemma} \label{lem:key}
For any non-empty $\Gamma_\rho$-invariant closed set $F \subset \P^1 \times \P^n$, we have $(\gamma^+,\chi_A) \in F$. 
\end{lemma}

\begin{proof}
Since $\P^1$ is compact, $F$ projects on a $\rho(\Gamma)$-invariant closed subset of $\P^n$ containing $L(\rho(\Gamma))$. It follows that there exists $\xi \in \P^1$ such that $(\xi,\chi_A) \in F$. 
If $\xi \neq \lim_{k \to +\infty} \gamma^{-k}(z)$, then 
$\lim_{k \to +\infty} \gamma^k(\xi) = \gamma^+$ and hence 
$$
 \lim_{k \to +\infty} \big(\gamma^k(\xi),\rho(\gamma^k)\chi_A \big) = \lim_{k \to +\infty}  \big(\gamma^k(\xi),A^k\chi_A \big) =  (\gamma^+,\chi_A) \in F.
$$
Otherwise, by the irreducibility condition {\rm (CG1)}, there exists $\gamma' \in \Gamma - <\gamma>$ such that $\rho(\gamma')\chi_A$ does not belong to the projection $\widebar{W}_A$ of $W_A$ into $\P^n$. Since 
$\Gamma$ is discrete, we have $\gamma'(\xi) \neq \xi$. As a consequence, we have: 
$$
\lim_{k \to +\infty} \big(\gamma^k(\gamma'(\xi)), A^k\rho(\gamma')\chi_A \big) = (\gamma^+,\chi_A) \in F. \qedhere
$$
 \end{proof}

\begin{proof}[Proof of the Theorem~\ref{thm:thmA*}]
By Lemma~\ref{lem:key}, the intersection of all  non-empty closed $\Gamma_\rho$ sets contains 
$$
\M = \overline{\{ \, (\gamma^+,\chi_A) \, | \, \gamma \in \Gamma , A = \rho(\gamma) \text{ proximal} \,\}} \subset L(\Gamma) \times L(\rho(\Gamma)).
$$
Thus $\M$ becomes the unique minimal set for the $\Gamma_\rho$-action on $\P^1 \times  \P^n$.
\end{proof}

\begin{remark}[\bf on the shape of $\M$] \label{rem:noninjective}
(1) If $\rho$ is not injective, then $\M = L(\Gamma) \times L(\rho(\Gamma))$ because $N = Ker \, \rho$ is a normal subgroup of $\Gamma$ and then $L(\Gamma) = L(N)$.
\medskip 

\noindent
(2) A similar conclusion holds if $\rho$ is indiscrete (in the sense that $\rho(\Gamma)$ is not discrete). Indeed, let $\gamma_k$ be a non-stationary sequence of elements of $\Gamma$ such that $\rho(\gamma_k) \to Id$.  Passing to a subsequence if necessary, there exist two points $\xi^-$ and $\xi^+$ in $\P^1$ such that 
$$\lim_{k \to +\infty} \gamma_k(\xi) = \xi^+$$ 
for any $\xi \neq \xi^-$ (see for example \cite[Lemma 2.2]{Alcalde-Dal'Bo}). For any $\chi \in L(\rho(\Gamma))$, take an element $(\xi,\chi) \in \M$ such that $\xi \notin \Gamma\xi^-$. Since $\Gamma$ is non elementary, such element always exists. For any $\gamma \in \Gamma$, we have: 
$$
 \lim_{k \to +\infty} \big(\gamma\gamma_k\gamma^{-1}(\xi),\rho(\gamma\gamma_k\gamma^{-1})\chi \big) = (\gamma(\xi^+),\chi). 
$$
Therefore $\Gamma_\rho \big(\gamma(\xi^+),\chi \big) \subset \M$ and hence $L(\Gamma) \times L(\rho(\Gamma)) = \M$.
\medskip 

\noindent
(3) In the opposite side, if $n=1$ and $\rho$ is the natural inclusion of $\Gamma$ into $\PSL$, then $\M = \{ \, (\xi,\xi) \, | \, \xi \in L(\Gamma)\, \}$.
\end{remark}

Two lemmas are needed to prove Theorem~\ref{thm:thmB*}:
 
\begin{lemma} \label{lem:twohyperbolic} 
Let $\Gamma$ be a non-elementary Fuchsian group and $\rho : \Gamma \to \PSLn$ be a representation satisfying {\rm (CG1)} and {\rm (CG2)}.
There are two hyperbolic elements $\gamma_1$ and $\gamma_2$ of $\Gamma$ such that
\medskip 

\noindent
(1) the dominant eigenvalues $\lambda_1$ and $\lambda_2$ generate a dense subgroup of the positive multiplicative group $\R^*_+$, 
\medskip 

\noindent
(2) $A_1= \rho(\gamma_1)$ and $A_2= \rho(\gamma_2)$ are proximal. 
\end{lemma}

\begin{proof} 
Under conditions {\rm (CG1)} and {\rm (CG2)}, the group $\rho(\Gamma)$ contains two elements $A_1$ and $A_2$ which generate a non-abelian free group containing only proximal elements (see \cite[Lemma 3.9]{Benoist1997} and \cite[Lemma 3]{Guivarc'h1990}). Let  $\gamma_1$ and $\gamma_2$ be two elements of $\Gamma$ such that $\rho(\gamma_1)=A_1$ and $\rho(\gamma_2)=A_2$. Reasoning as in  \cite{Dal'Bo1999}, we can replace $\gamma_1$ and $\gamma_2$ with two hyperbolic elements of $\Gamma$ whose dominant eigenvalues $\lambda_1$ and $\lambda_2$ generate a dense subgroup of $\R^*_+$.
\medskip 
\end{proof}

For each hyperbolic element $\gamma$ of $\Gamma$, we denote $v_\gamma$ the unit eigenvector in $E$ associated to dominant eigenvalue $\lambda_\gamma$. Clearly $v_\gamma \in E(\Gamma)$ since its direction $\bar v_\gamma = \gamma^+ \in L(\Gamma)$. From Theorem~\ref{thm:thmA*}, it follows that
$$
E(\M) \subset E(\Gamma) \times L(\rho(\Gamma)).
$$

\begin{lemma}  \label{lem:Gammaorbit}
Let $(v,\chi) \in E(\M)$ such that $\overline{\Gamma v} = E(\Gamma)$. For any hyperbolic element $\gamma \in \Gamma$ such that $A = \rho(\gamma)$ is proximal, there exists $\alpha \in \R^*$ such that 
$$(\alpha v_\gamma, \chi_A) \in \overline{\Gamma_\rho(v,\chi)}.$$
\end{lemma}

\begin{proof}
Assuming $\overline{\Gamma v}= E(\Gamma)$, there exists a sequence of elements $\gamma_k \in \Gamma$ such that the norms $\norm{\gamma_k v}$ converge to $0$ as $k \to +\infty$. Since $\Gamma$ is non elementary and $\rho(\Gamma)$ is irreducible, replacing $\gamma_k$ by $\gamma'\gamma_k$ for some $\gamma' \in \Gamma$, up to take a subsequence, we can suppose: 
\medskip 

\noindent
(1) $\gamma_kv = a_k v_\gamma + b_k v_{\gamma^{\tiny -1}}$ where $a_k \neq 0$ for any $k$, 
\medskip 

\noindent
(2) $\rho(\gamma_k)\chi \to \chi_0 \notin \widebar{W}_A$. 
\medskip 

\noindent
Let $p_k$ an increasing sequence of integers converging to $+\infty$ such that 
$\lambda_\gamma^{p_k}a_k$ converges to some $\alpha \neq 0$. Then we have 
$\gamma^{p_k}\gamma_kv \to \alpha v_\gamma$. Let us prove that 
\begin{equation} \label{eq:convergence1}
A^{p_k}\rho(\gamma_k) \chi \to \chi_A.
\end{equation}
Since $\chi_0 \notin \widebar{W}_A$, there exist an open neighbourhood $V(\chi_A)$ of $\chi_A$ containing $\chi_0$, an integer $N \gg 0$ and a constant $0 < c < 1$ satisfying \cite[Lemma 3]{Guivarc'h1990}:
\medskip 

\noindent
i) $A^{Nk}(V(\chi_A)) \subset V(\chi_A)$ for all $k \geq 0$, 
\medskip 

\noindent
iii) 
$\delta(A^{Nk}\chi_1,A^{Nk}\chi_2) \leq c^k \delta (\chi_1,\chi_2)$
for all $\chi_1, \chi_2 \in V(\chi_A)$ and for all $k \geq 0$, 
\medskip 

\noindent
For $k \geq 0$ large enough, we have $\rho(\gamma_k)\chi \in V(\chi_A)$. Assuming $p_k = Nq_k + r_k$ with $0 \leq r_k < N$, the inequality
$$
\delta(A^{Nq_k}\rho(\gamma_k)\chi,\chi_A) \leq c^{q_k} \delta(\rho(\gamma_k)\chi, \chi_A)
$$
implies 
$$
 \lim_{k \to +\infty} \delta(A^{Nq_k}\rho(\gamma_k)\chi,\chi_A) = 0
$$
and hence 
$$
 \lim_{k \to +\infty} \delta(A^{p_k}\rho(\gamma_k)\chi,\chi_A) = 0.
$$
This proves \eqref{eq:convergence1}. Finally, we deduce:
$$
 \lim_{k \to +\infty} \big( \gamma^{p_k}\gamma_kv,A^{p_k}\rho(\gamma_k)\chi\big) = (\alpha v_\gamma,\chi_A) \in \overline{\Gamma_\rho(v,\chi)}. \qedhere
$$
\end{proof}

\begin{proof}[Proof of the Theorem~\ref{thm:thmB*}]
Let $(v,\chi) \in E(\M)$ with $\overline{\Gamma v} = E(\Gamma)$. Take $\gamma_1, \gamma_2 \in \Gamma$ given by Lemma~\ref{lem:twohyperbolic} and its images $A_1 = \rho(\gamma_1)$ and $A_2 = \rho(\gamma_2)$. Applying Lemma~\ref{lem:Gammaorbit}, there exists a real number $\alpha_1 \neq 0$ such that 
$(\alpha_1 v_{\gamma_1}, \chi_{A_1}) \in \overline{\Gamma_\rho(v,\chi)}$
and hence
\begin{equation} \label{eq:product1}
(\alpha_1 \lambda_1^p v_{\gamma_1}, \chi_{A_1}) \in \overline{\Gamma_\rho(v,\chi)}
\end{equation}
for any $p \in \Z$. Since $\overline{\Gamma v}_{\gamma_1} = E(\Gamma)$ \cite[Theorem V.3.1]{Dal'Bobook}, by the same argument, we obtain another real number $\alpha_2 \neq 0$ such that 
\begin{equation} \label{eq:product2}
(\alpha_1 \alpha_2 \lambda_1^p \lambda_2^q v_{\gamma_2}, \chi_{A_2}) \in \overline{\Gamma_\rho(v,\chi)}
\end{equation}
for any pair $p,q \in \Z$.  As $\lambda_1$ and $\lambda_2$ generate a dense subgroup of $\R^*_+$ by Lemma~\ref{lem:twohyperbolic}, we deduce from \eqref{eq:product1} and \eqref{eq:product2} that
$$
(\lambda v_{\gamma_2}, \chi_{A_2}) \in \overline{\Gamma_\rho(v,\chi)}
$$
for any $\lambda > 0$. 
\medskip 

For any $(v',\chi') \in E(\M)$, since $(\bar v', \chi')$ and $(\gamma_2^+,\chi_{A_2})$ belong to the minimal set $\M$, there exists a sequence $\gamma_k \in \Gamma$ such that 
$$
\gamma_k \gamma_2^+ \to \bar v' \quad \text{and} \quad \rho(\gamma_k)\chi_{A_2} \to \chi'.
$$
It follows there exists a sequence $\lambda_k \in \R$ such that 
\begin{equation} \label{eq:convergence2}
\lambda_k \gamma_k v_{\gamma_2} \to v'. 
\end{equation}
As $(\lambda_k v_{\gamma_2}, \chi_{A_2}) \in \overline{\Gamma_\rho(v,\chi)}$, we have $(v',\chi') \in \overline{\Gamma_\rho(v,\chi)}$. 
\end{proof}

We deduce from Theorem~\ref{thm:thmB*} that $E(\M)$ is a non-empty minimal  $\Gamma_\rho$-invariant closed set if and only if $E(\Gamma)$ is a minimal $\Gamma$-invariant closed set. Since this condition is satisfied if and only if $\Gamma$ is convex-compact \cite[Proposition V.4.3]{Dal'Bobook}, we retrieve Corollary~\ref{cor:finitelygenerated}:

\begin{corollary} 
The set $\M_B$ is $U$-minimal if and only if
 $\Gamma$ is convex-cocompact 
\end{corollary}

More generally, since $\P^n$ is compact, any non-empty minimal $\Gamma_\rho$-invariant closed subset $F \subset E \times \P^n$ projects onto a non-empy minimal $\Gamma$-invariant closed subset $p_1(F) \subset E$. If $\Gamma$ is finitely generated, then either $F$ projects onto a closed $\Gamma$-orbit or $F = E(\M)$ and $\Gamma$ is convex-compact \cite[Theorem V.4.1]{Dal'Bobook}. On the contrary, if $\Gamma$ is not finitely generated, there exist examples where $E(\Gamma)$ does not admit any non-empty minimal $\Gamma$-invariant closed subset \cite{Kulikov,MatsumotoMinimal}.

\begin{corollary} 
There exist infinitely generated Fuchsian groups $\Gamma$ such that for any representations $\rho : \Gamma \to \PSLn$ satisfying conditions {\rm (CG1)} and {\rm (CG2)}, the projective bundle $Y$ does not admit any non-empty $U$-minimal subset of $\pi^{-1}(\Omega_X)$. 
\end{corollary}

\section{Proof of Theorem~\ref{thm:thmC}} \label{SthmC}

In this section, we restrict our attention to the space
$$
Y_{prox} = \Gamma_\rho \bs \PSL \times L(\rho(\Gamma)).
$$ 
This space is a $\PSL$-invariant closed subset of $Y$ for which the induced $\PSL$-action is minimal. From a geometrical point of view, $Y_{prox}$ is the unit tangent bundle to a minimal lamination by hyperbolic surfaces. Intersecting with $\pi^{-1}(\Omega_X)$, we obtain a $B$-invariant closed set 
$$
\Omega_{prox} = Y_{prox} \cap \pi^{-1}(\Omega_X) 
$$ such that: 
\medskip 

\noindent
(i) $\Omega_{prox}$ is included in the non-wandering set for the $U$-action on $Y_{prox}$, 
\medskip 

\noindent
(ii) $\Omega_{prox}$ inherits from $Y$ a natural structure of $L(\rho(\Gamma))$-fibre bundle over $\Omega_X$ with projection $\pi  : \Omega_{prox} \to \Omega_X$. 
\medskip 

\noindent
By duality, $U$-orbits in $\Omega_{prox}$ are in one-to-one correspondance with $\Gamma_\rho$-orbits in $E(\Gamma) \times L(\rho(\Gamma)$. Note that $\M_B \subset \Omega_{prox}$ is the unique non-empty minimal $B$-invariant closed subset of $\Omega_{prox}$. 
\medskip 

We also add a new condition on the representation $\rho : \Gamma \to \PSLn$, which we call \emph{Nielsen's condition}: 
\medskip

\noindent
{\rm (N)} there exists a continuous map  
$$
\varphi : L(\Gamma) \to L(\rho(\Gamma)),
$$
called \emph{limit map}, such that $\varphi \scirc \gamma = \rho(\gamma) \scirc \varphi$ for all $\gamma \in \Gamma$. 
\medskip 

\noindent
Conditions {\rm (CG1)}, {\rm (CG2)} and {\rm (N)} imply $\rho$ is discrete injective and $\varphi$ is surjective. 
\medskip 

A wide familiy of representations $\rho$ satisfying conditions {\rm (CG1)}, {\rm (CG2)} and {\rm (N)} can be find in the litterature: for $\rho(\Gamma) \subset SO(n,1)$ see \cite{Tukia1995} and for $\rho(\Gamma) \subset SL(n+1,R)$ Anosov see \cite{Labourie2006}. In general, even if $\Gamma$ is finitely generated, $\varphi$ is not  necessarily injective. This is the case for example if $\gamma$ is hyperbolic and $\rho(\gamma)$ is parabolic \cite{Tukia1988}. One of the most surprising examples is a discrete faithful representation $\rho : \Gamma \to SO(3,1)$ of a torsion-free cocompact Fuchsian group $\Gamma$ that gives raise to sphere-filling map $\varphi : S^1 \to S^2$ called the \emph{Cannon-Thurston map} \cite{CannonThurston}.

\begin{proof}[Proof of the Theorem~\ref{thm:thmC}]
Assume $\Gamma$ is non-elementary and $\rho : \Gamma \to \PSLn$ satisfies conditions {\rm (CG1)}, {\rm (CG2)} and {\rm (N)}. 
Under condition {\rm (N)}, we can immediately deduce the two following facts: 
\medskip 

\noindent
(i) the graph of map $\varphi$ is a non-empty $\Gamma_\rho$-invariant closed subset of $\P^1 \times \P^n$,
\medskip 

\noindent
(ii) the map $\varphi$ define a continuous section $\Phi : \Omega_X \to \Omega_{prox}$ given by 
$$
\Phi(\Gamma u ) = \Gamma_\rho (u,\varphi(u(+\infty))).
$$

\noindent
(1) The unique $B$-minimal set $\M_B$ is given by 
$$
\M_B = \Phi(\Omega_X) = \{ \, \Gamma_\rho(u, \varphi(u(+\infty))) \in Y \, | \, u \in \PSL, u(+\infty) \in L(\Gamma) \, \}.
$$
Indeed, by Theorem~\ref{thm:thmA*}, we know that the unique $\Gamma$-minimal set $\M \subset \P^1 \times \P^n$ coincides with the graph of $\varphi$. Now our statement follows by duality. 
\medskip 

\noindent
(2) The unique $B$-minimal set $\M_B$ is a $U$-attractor relative to $\Omega_{prox}$. Indeed,  take $y = \Gamma_\rho(u,\chi) \in \Omega_{prox}$
and assume $h_{s_k}(y) \to y'$ for some sequence $s_k \to \pm \infty$. Since $\Omega_{prox}$ is $U$-invariant, $y' = \Gamma_\rho(u',\chi')$ with $u'(+\infty) \in L(\Gamma)$ and  $\chi' \in L(\rho(\Gamma))$. As $y \in \M_B$ implies $y' \in \M_B$, the proof reduces to the case where $\chi = \varphi (\xi) \neq \varphi(u(+\infty))$. By construction, there exists a sequence $\gamma_k \in \Gamma$ such that 
$$
\gamma_k u \matriz{1}{s_k}{0}{1} \to u' \quad \text{and} \quad \rho(\gamma_k)\chi \to \chi'
$$
Let us return to the hyperbolic point of view, identifying $\PSL$ with the unit tangent bundle $T^1\H$. In this model, each element $u\in \PSL$ identifies with $u = (u(0),\vec{u}) \in T^1 \H$ where $u(0)$ is a point of $\H$ and $\vec{u}$ is a unit tangent vector to $\H$ at $u(0)$. 
Denoting $B_{u(+\infty)}(i,u(0))$ the Busemann cocycle centred at $u(+\infty)$ and calculated at $i$ and $u(0)$, we have the following conditions \cite[Chapter V]{Dal'Bobook}: 
\medskip 

\noindent
(a) $\gamma_k(u(+\infty)) \to u'(+\infty)$, 
\medskip 

\noindent
(b) $B_{\gamma_k(u(+\infty))}\big(i,\gamma_k (u(0))\big) \to B_{u'(+\infty)}\big(i,u'(0)\big)$, 
\medskip 

\noindent
(c) $\rho(\gamma_k)\chi \to \chi'$.
\medskip 

\noindent
Properties (a) and (b) imply  
$$
\lim_{k \to +\infty} \gamma_k(u(0)) = u'(+\infty). 
$$
Since 
$$
 B_{\gamma_k(u(+\infty))}\Big(i,\gamma_k\big(u(0)\big)\Big) = 
B_{u(+\infty)}\big(\gamma_k^{-1}(i),u(0)\big), \medskip
$$
applying again Property (b), we deduce:
$$
\lim_{k \to +\infty} \gamma_k^{-1}(i) = u(+\infty).
$$
As a consequence, since $\xi \neq u(+\infty)$, we have (see \cite[Lemma 2.2]{Alcalde-Dal'Bo}): 
$$
\gamma_k(\xi) \to u'(+\infty).
$$
By continuity of $\varphi$, it follows:
$$
 \rho(\gamma_k)\chi = \varphi(\gamma_k(\xi)) \to \varphi(u'(+\infty)). \medskip
$$
Property (c) implies $\chi' = \varphi(u'(+\infty))$ and hence $y'  \in \M_B$. 
\end{proof}

 
\begin {remark} Example~\ref{ex:lorentz} shows that we cannot expect $\M_B$ to be a global $U$-attractor in general. This is why we introduced the laminated space $Y_{prox}$. 
\end{remark}

\begin{proof}[Proof of the Corollary~\ref{cor:onetoone}] 
Let $m$ be a $U$-invariant (non necessarily finite) Radon measure on $\Omega_{prox}$. If $m$ is conservative, then Poincar\'e's Recurrence Theorem (see \cite[Theorem 1.1.5]{Aaronson} for the discrete version) implies that the set of $U$-recurrent points 
$$
\mathcal{R}_{prox} = \{ \, y \in \Omega_{prox}  \, | \, \exists \, s_n \to +\infty \, : \; h_{s_n}(y) \to y \,  \}
$$
has full-measure, that is, $m(\Omega_{prox} - \mathcal{R}_{prox}) = 0$. Since $\M_B$ is a $U$-attractor relative to $\mathcal{R}_{prox}$, then $\mathcal{R}_{prox} \subset \M_B$ and therefore $m(\Omega_{prox} - \M_B) = 0$. As the continuous section $\Phi$ sends homeomorphically $\Omega_X$ onto $\M_B$ and $m$ is supported by $\M_B$, the push-forward $\mu = \pi_\ast m$ is a $U$-invariant measure $\mu$ on $\Omega_X$. It is also conservative and verifies $\Phi_\ast \mu = m$. Finally $m$ is ergodic if and only if $\mu$ is ergodic. 
\end{proof}

If $\Gamma$ is finitely generated, as we recall in the introduction, any ergodic $U$-invariant measure $\mu$ on $\Omega_X$ either is supported by a closed orbit or is equal to the Burger-Roblin measure \cite{Burger,Roblin} up to a multiplicative constant. In the last case, $\mu$ is conservative, so Corollary~\ref{cor:uniqerg} follows from Corollary~\ref{cor:onetoone}. Namely, under the conditions of Theorem~\ref{thm:thmC} and assuming that $\Gamma$ is finitely generated, there is a unique conservative ergodic $U$-invariant Radon measure $m$ on $\Omega_{prox}$ (defined up to a multiplicative constant and supported by the unique $U$-minimal set $\M_B$ in $\Omega_{prox}$) if and only if $\Gamma$ is convex-cocompact. In particular, there is a unique $U$-invariant probability measure $m$ on $Y_{prox}$ if and only if $\Gamma$ is cocompact. 
Notice that the unique $U$-invariant probability measure on $Y$ (which is obtained by lifting the Haar measure) in \cite[Corollary~2.4]{Bonatti-Eskin-Wilkinson} is supported by $Y_{prox}$. In the counterexample constructed by S. Matsumoto \cite{MatsumotoWeak} on a 3-dimensional compact solvmanifold, there is a unique $B$-invariant probability measure $m$ supported by the unique $B$-minimal set $\M_B$, but there are uncountable many $U$-invariant probability measures (specifically, the ergodic components of $m$) supported by uncountable many $U$-minimal sets.

\end{document}